\pdfoutput=1

\documentclass[reqno]{amsart}

\usepackage{lmodern}
\usepackage[T1]{fontenc}
\usepackage[utf8]{inputenc}
\usepackage[english]{babel}
\usepackage{microtype} 

\usepackage{lineno} 

\usepackage{amsmath,amssymb,amsthm,amscd,mathrsfs,latexsym,mathtools,mathdots,enumerate,tikz}

\usepackage{ytableau}

\usetikzlibrary{positioning}

\usepackage[enableskew,vcentermath]{youngtab}
\usepackage[capitalize]{cleveref}
\usepackage{bm}
\usepackage{todonotes}
\presetkeys%
    {todonotes}%
    {inline,backgroundcolor=yellow}{}

\usepackage{graphicx}

\usepackage{ifpdf}
\graphicspath{{./figures/}}
\usepackage{epstopdf}
\DeclareGraphicsExtensions{.png,.jpg,.eps,.epsf,.pdf}

\newtheorem{theorem}{Theorem}
\newtheorem{proposition}[theorem]{Proposition}
\newtheorem{lemma}[theorem]{Lemma}
\newtheorem{corollary}[theorem]{Corollary}

\theoremstyle{definition}
\newtheorem{definition}[theorem]{Definition}
\newtheorem{example}[theorem]{Example}
\newtheorem{remark}[theorem]{Remark}
\newtheorem{question}[theorem]{Question}

\newcommand{\defin}[1]{{\color{blue}\emph{#1}}}

\newcommand{\setN}{\mathbb{N}}

\newcommand{\setQ}{\mathbb{Q}}

\newcommand{\xvec}{\mathbf{x}}

\newcommand{\symS}{S}

\newcommand{\stBruhat}{{\mathrm{st}}}
\newcommand{\lex}{{\mathrm{lex}}}

\newcommand{\kfK}{K}

\newcommand{\revsort}[1]{\boldsymbol{\lambda}(#1)}
\newcommand{\diagram}[1]{{\mathcal{#1}}}

\newcommand{\atom}{\mathcal{A}}
\newcommand{\hlPolyP}{\mathrm{P}}
\newcommand{\key}{\mathcal{K}}

\newcommand{\macdonaldP}{\mathrm{P}}
\newcommand{\nonSymE}{\mathrm{E}}

\newcommand{\schurS}{\mathrm{s}}

  \newcommand{\tpi}{\tilde{\pi}}
  \newcommand{\ttheta}{\tilde{\theta}}
  \newcommand{\trho}{\tilde{\rho}}

\newcommand{\NAF}{\mathrm{NAF}}

\DeclareMathOperator{\length}{\ell}

\DeclareMathOperator{\id}{id}

\DeclareMathOperator{\leg}{leg}
\DeclareMathOperator{\arm}{arm}
\DeclareMathOperator{\inv}{inv}
\DeclareMathOperator{\maj}{maj}
\DeclareMathOperator{\coinv}{coinv}
\DeclareMathOperator{\twinv}{twinv}

\DeclareMathOperator{\Des}{\mathrm{Des}}

\title{Non-symmetric Macdonald polynomials and Demazure--Lusztig operators}
\author{Per Alexandersson}
\address{Dept. of Mathematics, University of Pennsylvania. Philadelphia, PA}
\email{per.w.alexandersson@gmail.com}


\begin{document}

%

\begin{abstract}

We extend the family non-symmetric Macdonald polynomials and define \emph{permuted-basement Macdonald polynomials}.
We show that these also satisfy a triangularity property with respect
to the monomials bases and behave well under the Demazure--Lusztig operators.

The symmetric Macdonald polynomials $\macdonaldP_\lambda$ are expressed as
a sum of permuted-basement Macdonald polynomials via an explicit formula.

By letting $q=0$, we obtain $t$-deformations of key polynomials and Demazure atoms and 
we show that the Hall--Littlewood polynomials expand positively into these.
This generalizes a result by Haglund, Luoto, Mason and van Willigenburg.

As a corollary, we prove that Schur polynomials decompose with non-negative coefficients
into $t$-deformations of general Demazure atoms and thus generalizing the $t=0$ case which was previously known.
This gives a unified formula for the classical expansion of Schur polynomials in Hall--Littlewood polynomials
and the expansion of Schur polynomials into Demazure atoms.
\end{abstract}

\maketitle


\section{Introduction}

We study a generalization of non-symmetric Macdonald polynomials by adding a permutation parameter $\sigma$
to the combinatorial model for the classical non-symmetric Macdonald polynomials.
These are called \defin{permuted-basement Macdonald polynomials} and were previously introduced in \cite{Ferreira2011}.
The parameter $\sigma$ allows us to interpolate between two different parametrizations of
the Macdonald polynomials. This makes some unpublished results by J. Haglund and M. Haiman mentioned in
the introduction of \cite{Haglund2012,Remmel2013} explicit.

This extended family of polynomials satisfies many properties shared with
the classical non-symmetric Macdonald polynomials:

\begin{itemize}
 \item For each fixed value of $\sigma$, a triangularity property with respect to expansion in the monomial basis holds.
Consequently, the polynomials constitute a basis for $\setQ(q,t)[x_1,\dotsc,x_n]$ for each fixed $\sigma$.

 \item The permuted-basement Macdonald polynomials behave nicely under some affine Hecke algebra operators.
These operators are known as the Demazure--Lusztig operators,
which can be seen as a $t$-interpolation between the Demazure operators and the
operations that perform a simple transposition on indices of variables.

In particular, these operators act on the parameter $\sigma$ in a simple way,
see \cref{prop:basementPermutation}.
Consequently, there is a combinatorial definition based on fillings of diagrams,
as well as a recursive definition via such operators.

\item We give the expansion of the classical symmetric Macdonald polynomial, $\macdonaldP$, in the permuted-basement
Macdonald polynomials in \cref{thm:macdonaldPinGenMacE}. 

\item The specialization $q=0$ gives $t$-deformed Demazure atoms.
In particular, in \cref{cor:HLinGenAtoms2} we show that the Hall--Littlewood polynomials expands positively
in permuted-basement Macdonald polynomials when $q=0$, thus extending a result in \cite{Haglund2011463}.

\item The specialization $t=q=0$ of the permuted-basement Macdonald polynomials
give the Demazure characters (also known as key polynomials), and Demazure atoms.

\item The result in \cite[Prop. 6.1]{Mason2009} proves an equality between two combinatorial models for the key polynomials.
In \cref{prop:tKeyAsPBF}, we extend the result to incorporate the $t$ parameter
as well as showing then analogous statement for Demazure $t$-atoms,
\end{itemize}

Our goal with this paper is therefore to give a unified treatment of
non-symmetric Macdonald polynomials and specializations of these, such as Demazure atoms, key polynomials,
and operators acting on these.
The methods we use are based on the general theory
of non-attacking fillings described in \cite{HaglundNonSymmetricMacdonald2008}.

\section{Preliminaries -- Fillings and statistics}

Let $\sigma = (\sigma_1,\dots,\sigma_n)$ be a list of $n$ different positive integers and
let $\alpha=(\alpha_1,\dots,\alpha_n)$ be a weak integer composition, that is, a vector with non-negative integer entries.
An \defin{augmented filling} of shape $\alpha$ and \defin{basement} $\sigma$
is a filling of a Young diagram of shape $(\alpha_1,\dotsc,\alpha_n)$ with positive integers,
augmented with a zeroth column filled from top to bottom with $\sigma_1,\dotsc,\sigma_n$.

Note that we use \emph{English notation}, rather than the skyline fillings used in \cite{HaglundNonSymmetricMacdonald2008,Mason2009}.

\begin{definition}
Let $F$ be an augmented filling. Two boxes $a$, $b$, are \defin{attacking}
if $F(a)=F(b)$ and the boxes are either in the same column,
or they are in adjacent columns, with the rightmost box in a row strictly below the other box.
\begin{align*}
\begin{ytableau}
a  \\
\none[\scriptstyle\vdots] \\
b  \\
\end{ytableau}\quad
\text{or}\quad
\begin{ytableau}
 a & \none \\
 \none[\scriptstyle\vdots] \\
  & b \\
\end{ytableau}
\end{align*}
\end{definition}
A filling is \defin{non-attacking} if there are no attacking pairs of boxes.

\begin{definition}
A \defin{triple of type $A$} is an arrangement of boxes, $a$, $b$, $c$,
located such that $a$ is immediately to the left of $b$, and $c$ is somewhere below $b$,
and the row containing $a$ and $b$ is at least as long as the row containing $c$.

Similarly, a \defin{triple of type $B$} is an arrangement of boxes, $a$, $b$, $c$,
located such that $a$ is immediately to the left of $b$, and $c$ is somewhere above $a$,
and the row containing $a$ and $b$ is \emph{strictly} longer than the row containing $c$.

A type $A$ triple is an \defin{inversion triple} if the entries ordered increasingly,
form a \emph{counter-clockwise} orientation. Similarly, a type $B$ triple is an inversion triple
if the entries ordered increasingly form a \emph{clockwise} orientation.
If two entries are equal, the one with largest subscript in \cref{eq:invTriplets}
is considered largest.
\begin{equation}\label{eq:invTriplets}
\text{Type $A$:}\quad
\ytableausetup{centertableaux,boxsize=1.2em}
\begin{ytableau}
 a_3 & b_1 \\
 \none  & \none[\scriptstyle\vdots] \\
\none & c_2 \\
\end{ytableau}
\qquad
\text{Type $B$:}\quad
\ytableausetup{centertableaux,boxsize=1.2em}
\begin{ytableau}
c_2 & \none \\
\none[\scriptstyle\vdots]  & \none \\
a_3 & b_1 \\
\end{ytableau}
\end{equation}
\end{definition}

If $u = (i,j)$ let $d(u)$ denote $(i,j-1)$.
A \defin{descent} in $F$ is a non-basement box $u$ such that $F(d(u)) < F(u)$.
The set of descents inf $F$ is denoted $\Des(F)$.

\begin{example}
Below is a non-attacking filling of shape $(4,1,3,0,1)$ and with basement $(4,5,3,2,1)$.
The bold entries are descents, and the underlined entries form a type $A$ inversion triple.
There are $7$ inversion triples (of type $A$ and $B$) in total.
\[
\begin{ytableau}
\underline{4} & \underline{2} & 1 & \textbf{2} & 4\\
5 & 5\\
3 & 3 & \textbf{4} & 3\\
2\\
1 & \underline{1} \\
\end{ytableau}
\]
\end{example}

\medskip

The \defin{leg} of a box, denote $\leg(u)$, in an augmented diagram is the number of boxes to the right of $u$ in the diagram.
The \defin{arm}, denoted $\arm(u)$, of a box $u = (r,c)$ in an augmented diagram $\alpha$ is defined as the cardinality of
the set
\begin{align*}
\{ (r', c) \in \alpha : r < r' \text{ and } \alpha_{r'} \leq \alpha_r \} \cup \{ (r', c-1) \in \alpha : r' < r \text{ and } \alpha_{r'} < \alpha_r \}.
\end{align*}

We illustrate the boxes $x$ and $y$ (in the first and second set in the union, respectively) contributing to $\arm(u)$ below.
The boxes marked $l$ contribute to $\leg(u)$.
The $\arm$ values for all boxes in the diagram are shown in the diagram on the right.
\begin{equation*}
 \begin{ytableau}
\;  & y &   &   &  \\
  & y\\
 \\
  &   & \mathbf u & l & l  & l \\
  &   & x &  \\
  &  \\
  &   & x &   &  \\
 \end{ytableau}
\qquad\qquad
 \begin{ytableau}
\;  & 4 & 2 & 2 & 1\\
  & 1\\
 \\
  & 6 & 4 & 3 & 2 & 1\\
  & 3 & 1 & 0\\
  & 1\\
  & 4 & 3 & 1 & 1\\
 \end{ytableau}
\end{equation*}
The \defin{major index}, $\maj(F)$, of an augmented filling $F$ is given by
\begin{align*}
\maj(F) = \sum_{ u \in \Des(F) } \leg(u)+1.
\end{align*}
The \defin{number of inversions}, $\inv(F)$ of a filling is the number of inversion triples of either type.
The number of \defin{coinversions}, $\coinv(F)$, is the number of type $A$ and type $B$ triples which are \emph{not}
inversion triples.

\medskip

Let $\NAF_\sigma(\alpha)$ denote all non-attacking fillings of shape $\alpha$, augmented with the basement $\sigma \in \symS_n$,
and all entries in the fillings are in $\{1,\dotsc,n\}$.
\begin{example}
The set $\NAF_{3124}(1,1,0,2)$ consists of the following augmented fillings:
\begin{align*}
\substack{\young(31,12,2,443)\\ \coinv: 1\\ \maj: 1} \quad
\substack{\young(31,12,2,444)\\ \coinv: 1\\ \maj: 1} \quad
\substack{\young(32,11,2,443)\\ \coinv: 0\\ \maj: 0} \quad
\substack{\young(32,11,2,444)\\ \coinv: 0\\ \maj: 0}\\
\substack{\young(33,11,2,442)\\ \coinv: 1\\ \maj: 0} \quad 
\substack{\young(33,11,2,444)\\ \coinv: 0\\ \maj: 0} \quad 
\substack{\young(33,12,2,441)\\ \coinv: 2\\ \maj: 1} \quad 
\substack{\young(33,12,2,444)\\ \coinv: 0\\ \maj: 1}
\end{align*}

\end{example}

\section{A generalization of non-symmetric Macdonald polynomials}

The length of a permutation, $\length(\sigma)$, is the number of inversions in $\sigma$.
We use the standard convention and let $\omega_0$ denote the unique longest permutation in $\symS_n$,
that is, $\omega_0 = (n,n-1,\dotsc,1)$ in one-line notation.
Permutations act on compositions by permuting the entries.
Throughout the paper, $\alpha$ and $\gamma$ denote compositions, while $\lambda$ and $\mu$ are partitions.

\medskip

\begin{definition}
Let $\sigma \in \symS_n$ and let $\alpha$ be a composition with $n$ parts.
The \defin{non-symmetric permuted basement Macdonald polynomial} $\nonSymE^\sigma_\alpha(\xvec;q,t)$ is defined as
\begin{equation}\label{eq:nonSymmetricMacdonaldBasement}
\nonSymE^\sigma_\alpha(\xvec; q,t) = \sum_{ F \in \NAF_\sigma(\alpha)} \xvec^F q^{\maj F} t^{\coinv F} \!\!\!
\prod_{ \substack{ u \in F \\ F(d(u))=F(u) }} \!\!\! \frac{1-t}{1-q^{1+\leg u} t^{1+\arm u}},
\end{equation}
where $F(d(u)) \neq F(u)$ in the product index if $u$ is a box in the basement.
\end{definition}
When $\sigma = \omega_0$, we recover\footnote{There is a slight difference in notation, the shape index here is reversed, compared to \cite{HaglundNonSymmetricMacdonald2008}.}
the non-symmetric Macdonald polynomials defined in \cite{HaglundNonSymmetricMacdonald2008}, $\nonSymE_\alpha(\xvec;q,t)$.
We refer to this particular value for $\sigma$ as the \emph{key} basement and
we simply write $\nonSymE_\alpha(\xvec; q,t)$ for $\nonSymE^{\omega_0}_\alpha(\xvec; q,t)$.

\subsection{Properties of non-symmetric Macdonald polynomials}

The following relation is a part of the Knop--Sahi recurrence relations for Macdonald polynomials \cite{Knop1997,Sahi1996}:
\begin{equation}\label{eq:knopRelation}
 \nonSymE_{\hat{\alpha}}(\xvec;q,t) = q^{\alpha_1} x_1 \nonSymE_{\alpha}(x_2,\dotsc,x_n,q^{-1}x_1;q,t)
\end{equation}
where $\hat{\alpha} = (\alpha_2,\dotsc,\alpha_n, \alpha_1 +1)$.
Also note that
\[
\nonSymE^\sigma_{(\alpha_1+1,\dotsc,\alpha_n+1)}(\xvec;q,t) = (x_1\dotsm x_n) \nonSymE^\sigma_{(\alpha_1,\dotsc,\alpha_n)}(\xvec;q,t),
\]
which allow us to extend the definition of non-symmetric Macdonald ``polynomials'' for compositions $\alpha$ with negative entries.

\begin{proposition}[Corollary 3.6.4 in \cite{HaglundNonSymmetricMacdonald2008}]\label{prop:specToAtom}
We have the relation
\begin{equation*}
\nonSymE^{\omega_0}_\alpha(x_1,\dotsc,x_n; q,t) =\nonSymE^{\id}_\alpha(x_n,\dotsc,x_1; q^{-1}, t^{-1}).
\end{equation*}
\end{proposition}
The polynomials appearing in the right hand side above is a version of non-symmetric Macdonald polynomials that are studied in \cite{Marshall1999}.

\begin{question}
Is there a relation similar to that in \cref{prop:specToAtom} for other basements?
\end{question}

Using \cref{eq:nonSymmetricMacdonaldBasement} and \cref{prop:specToAtom}, we obtain the following
diagram of specializations, where we recover the \defin{key polynomial} $\key_\alpha(\xvec)$ and \defin{Demazure atom}, $\atom_\alpha(\xvec)$.
These specializations were proved in \cite{Mason2009} --- we give the classical
definition (as described in \cite{Lascoux1990Keys,ReinerShimozono1995}) of key polynomials and Demazure atoms in \cref{sec:tAtoms}.
\[
\begin{CD}
\nonSymE^{\sigma}_\alpha(\xvec;q,t) @>{\sigma=\omega_0}>> \nonSymE^{\omega_0}_\alpha(\xvec;q,t) @>{t=q=0}>> \key_\alpha(\xvec)  @>{\alpha = \lambda}>>  \schurS_\lambda(\xvec)\\
@VV{\sigma=\id}V @VV{\substack{t=q=\infty \\ x_i \to x_{n+1-i}}}V \\
\nonSymE^{\id}_\alpha(\xvec;q,t) @>{t=q=0}>> \atom_\alpha(\xvec)
\end{CD}
\]
It is also easy to verify that semi-standard augmented fillings of partition shape $\lambda$ and basement $\omega_0$
can be put into bijection with semi-standard Young tableaux of shape $\lambda$. This shows that
key polynomial $\key_\lambda$ specialize to the Schur polynomial $\schurS_\lambda$ (in $n$ variables) whenever $\lambda$ is a partition.

\bigskip 
The classical non-symmetric Macdonald polynomials specialize to other well-known families of polynomials:
\[
\begin{CD}
\nonSymE_\alpha(\xvec;q,t) @>{\ast}>> \macdonaldP_\lambda(\xvec;q,t) \\
 @VV{ \substack{ \alpha = \lambda \\ q=1 \\ t=0  } }V  @VV{ \substack{ q=1 \\ t=0  } }V \\
 e_{\lambda'}(\xvec) & & m_\lambda(\xvec).
\end{CD}
\]
Here, $\ast$ is indicate the relation 
\[
\nonSymE_{(\lambda,0^n)}(x_1,\dotsc,x_n,0,\dotsc,0;q,t) = \macdonaldP_\lambda(x_1,\dotsc,x_n;q,t)
\]
for partitions $\lambda$ with $n$ parts. 
The polynomials $\nonSymE_\alpha(\xvec;1,0)$ can thus be interpreted 
as a non-symmetric analogue of the elementary symmetric functions $e_\lambda$.


\bigskip 

Some specializations of non-symmetric Macdonald polynomials, such as $\nonSymE_\alpha(\xvec;q,0)$
and $\nonSymE_\alpha(\xvec;q^{-1},\infty)$ have representation-theoretical interpretations,
see \cite{FeiginMakedonskyi2015b,FeiginMakedonskyi2015}.
In particular, we note that \cite{FeiginMakedonskyi2015b} consider the combinatorial
model defined in \eqref{eq:nonSymmetricMacdonaldBasement} in an expansion of some $\nonSymE_\alpha(\xvec;q^{-1},\infty)$.

\subsection{Alcove walk model}

Another interesting article concerning non-symmetric Macdonald polynomials is \cite{RamYip2011},
which gives a combinatorial model using alcove walks. The basic idea is
to repeatedly use \cref{prop:thetaShapeTrans} below, expand the product
and interpret the terms. This method expresses a Macdonald polynomials as a sum over alcove walks,
starting at the \emph{fundamental alcove} and ending at the alcove representing
the particular Macdonald polynomial we are interested in.

It would be an interesting project to see if permuted-basement Macdonald polynomials can be generated in this way.
A natural conjecture is that the choice of a starting alcove --- which can be done in $n!$ ways ---
corresponds to the basement.

This interpretation would allow us to extend permuted-basement Macdonald
polynomials to other types. Note that the notion of key polynomials
is known in other types, defined via crystal operators, see for example \cite{HershLenart2015}.

\subsection{Triangularity}

In this subsection, we prove that the permuted-basement Macdonald polynomials
satisfy a triangularity property with respect to the monomial basis.

\begin{definition}
We define the \defin{Bruhat order} on compositions of $m$ with $n$ parts as the transitive closure of the following relations:
\begin{itemize}
\item If $i<j$ and $\alpha_j > \alpha_i$ then $\alpha >_\stBruhat s_{ij}(\alpha)$ where $s_{ij}$ is the transposition $(i,j)$.
\item If $i<j$ and  $\alpha_j- \alpha_i>1$ then $s_{ij}(\alpha) >_\stBruhat \alpha + e_i - e_j$.
\end{itemize}
\end{definition}

Just as for the standard non-symmetric Macdonald polynomials,
the $\nonSymE^\sigma_\alpha(\xvec; q,t)$ satisfy a triangularity condition with respect to the monomial basis:
\begin{equation}\label{eq:triangularity}
 \nonSymE^\sigma_\alpha(\xvec; q,t) \in \xvec^{\sigma^{-1}(\alpha)} + \setQ(q,t)\{ \xvec^{\sigma^{-1}(\gamma)} : \gamma <_\stBruhat \alpha \}.
\end{equation}
Alternatively, this can be expressed in a slightly more pleasant way as
\begin{equation}\label{eq:triangularityAlt}
 \nonSymE^\sigma_{\alpha}(\sigma\xvec; q,t) \in \xvec^{\alpha} + \setQ(q,t)\{ \xvec^{\gamma} : \gamma <_\stBruhat \alpha \}.
\end{equation}

We prove triangularity with respect to a lexicographic total ordering,
similar to what is done in \cite{Macdonald1995} for the classical symmetric Macdonald polynomials.
This ordering extends the Bruhat order defined above.
Here, $\revsort{\alpha}$ denotes the unique partition obtained from $\alpha$ by sorting the parts in a decreasing manner,
and $>_\lex$ is the standard lexicographic order, comparing elements componentwise from left to right.

\begin{proposition}[Triangularity]\label{prop:triangularity}
Let $\gamma$ and $\alpha$ be weak compositions of $m$ with $n$ parts.
Then for any basement $\sigma \in \symS_n$,
\[
 [\xvec^{\sigma^{-1}(\gamma)}] \nonSymE^\sigma_\alpha(\xvec; q,t) =
\begin{cases}
 0 \quad \text{ if } \revsort{\gamma} >_\lex \revsort{\alpha}, \\
 0 \quad \text{ if } \revsort{\gamma} = \revsort{\alpha} \text{ and } \gamma >_\lex \alpha,\\
 1 \quad \text{ if } \gamma=\alpha.
\end{cases}
\]
\end{proposition}
\begin{proof}
First note that
$
[x^{\sigma^{-1}(\gamma)}]\nonSymE^\sigma_\alpha(\xvec;q,t) =
[x^{\gamma_1}_{\sigma_1}x^{\gamma_2}_{\sigma_2} \dotsm x^{\gamma_k}_{\sigma_k}]\nonSymE^\sigma_\alpha(\xvec;q,t)
$,
so we focus on non-attacking fillings of shape $\alpha$ and $\gamma_i$ entries with value $\sigma_i$ for $i=1,\dotsc,n$.

\noindent
\textbf{Case $\revsort{\gamma} >_\lex \revsort{\alpha}$:}
Let $\lambda = \revsort{\gamma}$ and $\mu = \revsort{\alpha}$.
The condition implies that there is some $j\geq 1$ such that
\[
\lambda_1 = \mu_1,\quad \lambda_2 = \mu_2, \quad  \dotsc \quad \lambda_{j-1} = \mu_{j-1} \; \text{ and } \; \lambda_j > \mu_j.
\]
Suppose there is a way to create a non-attacking filling with these properties.
Then there must be $\lambda_1$ equal entries in different columns,
then $\lambda_2$ equal entries in different columns and so on.

If $j=1$, it is evident that there is no such non-attacking filling,
since $\lambda_1$ entries must appear in different columns but there are only $\mu_1 (<\lambda_1)$ columns available.

In the case $j>1$, it is straightforward to show by induction
that after placing the first $\lambda_1 + \lambda_2 + \lambda_{j-1}$
entries, the number of columns with available empty boxes is $\mu_j$.
Since $\mu_j<\lambda_j$, there is no non-attacking filling
with weight $x^{\gamma_1}_{\sigma_1}x^{\gamma_2}_{\sigma_2} \dotsm x^{\gamma_n}_{\sigma_n}$,
shape $\gamma$ and basement $\sigma$ if $\revsort{\gamma} >_\lex \revsort{\alpha}$.

\noindent
\textbf{Case $\revsort{\gamma} = \revsort{\alpha}$ and $\gamma >_\lex \alpha$:}
Assume that there is a filling $T$ with shape $\alpha$ and
weight $x^{\gamma_1}_{\sigma_1}x^{\gamma_2}_{\sigma_2} \dotsm x^{\gamma_n}_{\sigma_n}$.
Let $\gamma_i$ be a largest entry in $\gamma$. This implies that
there is exactly one entry $\sigma_i$ in each column of $T$ and in particular, at the end of some longest row
with length $\alpha_l = \gamma_i$. The non-attacking condition for adjacent columns
now implies that if column $c$ has an entry equal to $\sigma_i$ in row $r_1$,
and column $c+1$ has an entry equal to $\sigma_i$ in row $r_2$,
then $r_1 \geq r_2$. Hence $T$ is of the form exemplified in \cref{eq:triangularityProofExample}
where $\ast$ marks entries with value $\sigma_i$.
\begin{align}\label{eq:triangularityProofExample}
 \ytableausetup{centertableaux,boxsize=1.2em}
\begin{ytableau}
\sigma_1  &  &  &   & \none[\alpha_1]\\
\sigma_2  &  &  &  &   \ast&  \ast &  \ast & \none[\alpha_2]\\
\sigma_3  &  &  \ast&  \ast &   & \none[\alpha_3]\\
\sigma_4  & \ast & \none[\alpha_4]   \\
\sigma_5  &  &   & \none[\alpha_5]\\
\end{ytableau},
\end{align}
It follows that $i \geq l$. By removing the last box in row $l$, we obtain a smaller filling $T'$,
with weight and shape given by
\[
\gamma' = (\gamma_1,\dotsc,\gamma_{i-1},\gamma_i-1,\gamma_{i+1},\dotsc,\gamma_n),
\qquad \alpha' = (\alpha_1,\dotsc,\alpha_{l-1},\alpha_l-1,\alpha_{i+1},\dotsc,\alpha_n).
\]
Finally, note that $\revsort{\gamma'} = \revsort{\alpha'}$ and $\gamma' >_\lex \alpha'$,
since $\gamma_i = \alpha_l$ and $i \geq l$.

However, this is absurd, since repeating this procedure eventually yields the empty filling,
where $\gamma >_\lex \alpha$ is no longer true.
Therefore, there cannot be a valid filling $T$ satisfying all conditions to begin with.

\noindent
\textbf{Case $\gamma=\alpha$:}
As in the previous case, we suppose that there is a filling, $T$, satisfying these conditions,
and we repeatedly remove a box from a longest row.
This operation preserves the property $\gamma \geq_\lex \alpha$,
but we know that as soon as a strict inequality is obtained, there is no such filling.

In order to have equality $\gamma = \alpha$ after each removal of a box,
we need that all $\sigma_i$ appear in the same row.
It follows that there is a unique filling, where every row $i$ is filled with boxes with value $\sigma_i$.
This filling has no inversions and no two different horizontally adjacent boxes,
so $T$ contributes with the monomial $x^{\gamma_1}_{\sigma_1}x^{\gamma_2}_{\sigma_2} \dotsm x^{\gamma_n}_{\sigma_n}$.
This proves the triangularity statement in \eqref{eq:triangularity}.
\end{proof}

\begin{question}
Is there a natural inner product (depending on $\sigma$) for which
the $\nonSymE^\sigma_\alpha(\xvec; q,t)$ form an orthogonal basis?
\end{question}

\section{Demazure--Lusztig operators}

In this section we introduce a set of operators acting on polynomials in $x_1,\dotsc,x_n$.
These appear in in the study of key polynomials and Demazure atoms,
see \emph{e.g.}~ the paper \cite{ReinerShimozono1995} by V. Reiner and M. Shimozono for a background on key polynomials,
as well as properties of these operators.

Let $s_i$ be a simple transposition on indices of variables and define
\[
 \partial_i = \frac{1-s_i}{x_i-x_{i+1}}, \quad \pi_i = \partial_i x_i, \quad \theta_i = \pi_i - 1.
\]
Note that $\partial_i (f)$ is indeed a polynomial if $f$ is, since $f-s_i f$ is divisible by $x_i - x_{i+1}$.
The operators $\pi_i$ and $\theta_i$ are used to define the key polynomials and Demazure atoms, respectively,
and we give this definition further down.
It should be mentioned that $\theta_i$ and $\pi_i$ are closely related to \emph{crystal operators} and $i$-strings,
see \cite{Mason2009} for details.
Now define the following $t$-deformations of the above operators:
\begin{align}
\tpi_i(f) = (1-t)\pi_i(f) + t s_i(f)  \qquad \ttheta_i(f) = (1-t)\theta_i(f) + t s_i(f).
\end{align}
The $\ttheta_i$ are called the \defin{Demazure--Lusztig operators} and are
generators for the affine Hecke algebra that appear in \cite{HaglundNonSymmetricMacdonald2008} (where $\ttheta_i$ is denoted $T_i$).
A similar set of operators appear in \cite[p.4]{Lascoux97ribbontableaux}, in the definition of Hall--Littlewood functions.

\medskip 

It should be mentioned that \cite{Ferreira2011} provides a nice characterization the permuted-basement Macdonald polynomials
as simultaneous eigenfunctions of certain products of such operators and the operation in \eqref{eq:knopRelation}.
This is a generalization of \emph{Cherednik's representation} \cite{Cherednik1995} of the affine Hecke algebra mentioned above.

\subsection{Some properties of $\ttheta_i$ and $\tpi_i$}

Using the definition above, it is straightforward to show that $\ttheta_i^2 = (t-1) \ttheta_i + t$,
which implies that $\tpi_i \ttheta_i (f) = \ttheta_i \tpi_i (f) = t f$.
Hence, $\ttheta_i$ and $\tpi_i$ are essentially inverses of each other.
We also have that $\tpi_i$ can be expressed in $\ttheta_i$ as
\begin{equation}\label{eq:tpi-in-ttheta}
\tpi_i(f) = \ttheta_i(f) + (1-t)(f).
\end{equation}

\medskip

As for the $s_i$, the $\ttheta_i$ and $\tpi_i$ satisfy the braid relations:
\begin{itemize}
 \item $\ttheta_i \ttheta_j = \ttheta_j \ttheta_i$ whenever $|i-j| \geq 2$ and
\item $\ttheta_i \ttheta_j \ttheta_i = \ttheta_j \ttheta_i \ttheta_j$ when $|i-j|=1$.
\end{itemize}
The same relations hold for the $\tpi_i$.
This implies that if $\omega = \omega_1 \dotsm \omega_\ell$ and $\omega' = \omega'_1 \dotsm \omega'_\ell$
are both reduced words for the same permutation in $\symS_n$, then
$\ttheta_{\omega_1} \ttheta_{\omega_2}  \dotsb \ttheta_{\omega_\ell} = \ttheta_{\omega'_1}\ttheta_{\omega'_2}  \dotsb \ttheta_{\omega'_\ell}$.
Hence, if $\tau \in \symS_n$ is a permutation, we can define
$\ttheta_\tau$ as $\ttheta_{\omega_1} \ttheta_{\omega_2}  \dotsb \ttheta_{\omega_\ell}$
where $\omega$ is any reduced word for $\tau$.
The braid relations above ensure that this is \emph{independent} of the choice of reduced word and thus well-defined.
We define $\tpi_\tau$ in a similar fashion.

The following lemma is a straightforward consequence of the definitions above:
\begin{lemma}\label{lem:operatorSymmetry}
If $f$ is symmetric in $x_ix_{i+1}$ then 
\begin{itemize}
 \item $\ttheta_i(f) = tf$,
 \item $\tpi_i(f) = f$,
 \item $\ttheta_i(f \cdot g) = f \cdot \ttheta_i(g)$ for any $g$,
 \item $\ttheta_j f(\xvec)$ is symmetric in  $x_ix_{i+1}$ for $j \notin \{i-1,i+1\}$.
\end{itemize}
\end{lemma}

The following lemma is an important tool in \cref{subsec:basementPerm}.
\begin{lemma}[\cite{HaglundNonSymmetricMacdonald2008}]\label{lem:tpisym}
We have $\ttheta_i( f ) = g$ \emph{if and only if} $f+g$ and $tx_{i+1} f + x_i g$ are both symmetric in $x_ix_j$.
\end{lemma}

\begin{lemma}\label{lem:braidRelations}
The following \emph{mixed} braid relations holds for $\tpi_i$ and $\ttheta_i$:

\begin{align*}
\begin{cases}
\tpi_i \tpi_{i-1} \ttheta_i &= \ttheta_{i-1} \tpi_i \tpi_{i-1} \\
\tpi_{i-1} \tpi_i \ttheta_{i-1} &= \ttheta_{i} \tpi_{i-1} \tpi_{i}  
\end{cases}
\qquad \text{ and } \qquad
\begin{cases}
\tpi_{i-1} \ttheta_{i} \ttheta_{i-1} &= \ttheta_{i} \ttheta_{i-1} \tpi_i \\
\tpi_{i} \ttheta_{i-1} \ttheta_{i} &= \ttheta_{i-1} \ttheta_{i} \tpi_{i-1}.
\end{cases}
\end{align*}
\end{lemma}
\begin{proof}
Express $\tpi_i$ and $\tpi_{i-1}$ in terms of $\ttheta_i$ and $\ttheta_{i-1}$ respectively and expand.
\end{proof}

\subsection{Something about knots}

There is a deep connection between Macdonald polynomials and knot theory, 
see for example the connection between Jones polynomials and Macdonald polynomials \cite{Cherednik2012}.
It is not surprising, given the braid relations involved with $\ttheta_i$ and $\tpi_i$.
For a background on the braid group, see the introduction and definitions in \cite{Dehornoy2008}.
Intuitively, $\ttheta_i$ and $\tpi_i$ can be seen as $\hat{s}_i$ and $\hat{s}_i^{-1}$ in the Artin presentation of the braid group. 
The the relations in \ref{lem:braidRelations} are compatible with this interpretation,
the only caveat here is that $\ttheta_i \tpi_i = t$, while $\hat{s}_i \hat{s}_i^{-1} = \id$.
This has the consequence that if $\hat{s}_{i_1}^{\pm 1} \hat{s}_{i_2}^{\pm 1} \dotsm \hat{s}_{i_\ell}^{\pm 1}$
is a \emph{reduced word} in the braid group,
then substituting $\hat{s}_{i} \mapsto \ttheta_i$ and $\hat{s}^{-1}_{i} \mapsto \tpi_i$
gives a reduced word of operators. Furthermore, if $w_1$ and $w_2$ are \emph{reduced words} representing the same braid,
then the corresponding compositions of operators acts the same.

\subsection{Symmetries of diagram fillings}

In this subsection, we introduce the necessary notation to state an important
proposition proved in \cite{HaglundNonSymmetricMacdonald2008}.
The complete proof is fairly involved and closely related to the theory of LLT polynomials,
see the Appendix in \cite{qtCatalanBook}.

We generalize the notion of diagrams, arm values, leg values, major index and inversions.
A \defin{lattice-square diagram} $\diagram{D}$ is subset of boxes $(i,j) \in \setN^2$.
The \defin{reading order} of a lattice diagram is the total order given by reading 
squares column by column from right to left, and from top to bottom within each column,
as in \eqref{eq:readingOrderExample}.
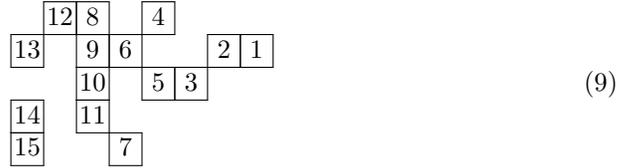
\begin{equation}\label{eq:readingOrderExample}
 \ytableausetup{centertableaux,boxsize=1.2em}
\begin{ytableau}
\none & 12 & 8 & \none  & 4 \\
 13 & \none & 9 & 6 & \none  & \none  & 2  &  1\\
\none  & \none & 10 &  \none & 5  & 3 \\
 14 & \none & 11   \\
15 & \none & \none  & 7 \\
\end{ytableau}
\end{equation}
Arm and leg values for each box in a lattice-square diagram are arbitrary fixed non-negative integers.
We say that two boxes $u$, $v$ form an \defin{inversion pair} in a filling $F$ if they are attacking,
$v$ precedes $u$ in reading order and $F(u)<F(v)$.
Similarly, the \defin{descent set}, $\Des F$, of a filling is the set of boxes $u \in F$
such that $d(u) \in F$ and $F(d(u))<F(u)$.
Define $\maj$ and $\inv$ statistics for arbitrary lattice-square fillings as
\begin{align}
\inv F &= |\{(u,v) : u, v \text{ form an inversion pair in } F \}| - \sum_{s \in \Des F} \arm s \\
\maj F &= \sum_{s \in \Des F} (1 + \leg s).
\end{align} 
It is shown in \cite{HaglundNonSymmetricMacdonald2008} that these definitions extends the corresponding statistics on augmented fillings.

The following powerful proposition appears in \cite[Prop. 4.2.5]{HaglundNonSymmetricMacdonald2008}
which is later used to determine symmetries of expressions obtained as a sum over non-attacking fillings.

\begin{proposition}\label{prop:latticeDiagramSymmetry}
Consider two disjoint lattice diagrams, $\diagram{S}$ and $\diagram{B}$ and two disjoint subsets $\diagram{Y},\diagram{Z} \subseteq \diagram{S}$.
Let $\hat{\diagram{S}} = \diagram{B} \cup \diagram{S}$ and suppose we have a
fixed filling $B: \diagram{B} \to [n]$ which does not contain the entries $i$ and $i+1$.
Fix arm and leg values for all boxes $u$ such that $u \in \diagram{S}$ and $d(u) \in \hat{\diagram{S}}$.
For any filling $F : \diagram{S} \to [n]$, set $\hat{F} = B \cup F$. Then the sum
\begin{align*}
 \sum_{ \substack{ F : \diagram{S} \to [n] \\ \hat{F} \text{ non-attacking} \\ 
 F(\diagram{Y}) \cap \{i, i+1\} = \emptyset \\ F(\diagram{Z}) \subseteq \{i,i+1\} }}
x^F q^{\maj \hat{F}} t^{-\inv \hat{F}}
\prod_{\substack{u \in \diagram{S},\; d(u)\in \hat{\diagram{S}}\\ \hat{F}(u) = \hat{F}(d(u)) }} \left( 1- q^{1+\leg u} t^{1+\arm u} \right)
\prod_{\substack{u \in \diagram{S} \\ \hat{F}(u) \neq \hat{F}(d(u)) }} (1-t)
\end{align*}
where in the last factor, $\hat{F}(u) \neq \hat{F}(d(u)) $ is considered to be true if $d(u) \notin \hat{\diagram{S}}$,
is symmetric in $x_i x_{i+1}$.
\end{proposition}
The fixed filling $B$ plays the rôle of a basement and 
the sets $\diagram{Z}$ and $\diagram{Y}$ are subsets of $\diagram{S}$
that specify which boxes should and should not contain $i$ and $i+1$.

\subsection{Permuting the basement}\label{subsec:basementPerm}

The following two properties in \cref{prop:basementPermutation} are essentially
inverses of each other --- we can use $\ttheta_i$ to decrease the length of the basement and $\tpi_i$ to increase the length.
The result in the following proposition appears without proof in \cite{Ferreira2011}, referencing a private communication
with J.~Haglund. We provide a full proof below.

\begin{proposition}[Basement permuting operators]\label{prop:basementPermutation}
Let $\alpha$ be a composition and let $\sigma$ be a permutation.
Furthermore, let $\gamma_i$ be the length of the row with basement label $i$, that is, $\gamma_i = \alpha_{\sigma^{-1}_i}$.

If $\length(\sigma s_i)  < \length(\sigma)$, then
\begin{align}\label{eq:atomOperatorOnBasement}
\ttheta_i \nonSymE_\alpha^{\sigma}(\xvec;q,t) = \nonSymE_\alpha^{\sigma s_i}(\xvec;q,t) \times
\begin{cases}
 t \text{ if } \gamma_i \leq \gamma_{i+1} \\
 1 \text{ otherwise.}
\end{cases}
\end{align}
Similarly, if $\length(\sigma s_i)  > \length(\sigma)$, then
\begin{align}\label{eq:keyOperatorOnBasement}
\tpi_i \nonSymE_\alpha^{\sigma}(\xvec;q,t) = \nonSymE_\alpha^{\sigma s_i}(\xvec;q,t) \times
\begin{cases}
 t \text{ if } \gamma_i < \gamma_{i+1} \\
 1 \text{ otherwise.}
\end{cases}
\end{align}
\end{proposition}
\begin{proof}
Using \cref{lem:tpisym}, to prove \eqref{eq:atomOperatorOnBasement}, it is enough to show two symmetries in $x_ix_{i+1}$.
There are two cases to consider:

\noindent
\textbf{Case $\gamma_i \leq \gamma_{i+1}$:}
It suffices to show that
\begin{equation}\label{eq:symmetriesCase1}
\nonSymE_\alpha^{\sigma}(\xvec;q,t) + t \cdot \nonSymE_\alpha^{\sigma s_i}(\xvec;q,t) \quad \text{ and } \quad
tx_{i+1} \cdot \nonSymE_\alpha^{\sigma}(\xvec;q,t) + t x_i \cdot \nonSymE_\alpha^{\sigma s_i}(\xvec;q,t)
\end{equation}
are symmetric in $x_i x_{i+1}$. We show these symmetries using \cref{prop:latticeDiagramSymmetry}.

For the first symmetry, let $\hat{\diagram{S}}$ be the augmented diagram with shape $\alpha$
and let $\diagram{B}$ be all boxes in the basement $\sigma$ not containing $\{i, i+1\}$.
Now consider non-attacking fillings $\hat{\sigma}: \hat{\diagram{S}} \to [n]$:
each such filling has $i$ and $i+1$ appearing in the basement column exactly once.
Thus, every such filling $\hat{\sigma}$ corresponds to either a filling
for $\nonSymE_\alpha^{\sigma}$ or $\nonSymE_\alpha^{\sigma s_i}$.
However, there is an extra inversion of type $A$ in the leftmost diagram in \eqref{eq:case11}, compared to the diagram on the right,
given by the boxes marked $\{i, i+1, \infty \}$. 
\begin{equation}\label{eq:case11}
\ytableausetup{centertableaux,boxsize=1.6em}
\begin{ytableau}
\infty & \scriptstyle{\mathbf{i+1}}  &  & \cdots & \\
\none & \none[\vdots]   \\
\none & \mathbf{i}  &  &  \\
\end{ytableau}
\quad
\text{ and }
\quad
\begin{ytableau}
\infty &\mathbf{i}  &  & \cdots & \\
\none&\none[\vdots]   \\
\none&\scriptstyle{\mathbf{i+1}}  &  &  \\
\end{ytableau}
\end{equation}
It follows that the sum over non-attacking fillings with basement $\sigma$ in \eqref{eq:case11}
is, up to a constant, $t^{-1} \nonSymE_\alpha^{\sigma}$,
and the sum over fillings with basement $\sigma s_i$ is, up to \emph{the same} constant, $\nonSymE_\alpha^{\sigma s_i}$.
\cref{prop:latticeDiagramSymmetry} states that the total sum $t^{-1} \nonSymE_\alpha^{\sigma} + \nonSymE_\alpha^{\sigma s_i}$
is symmetric in $x_i x_{i+1}$ which implies the first symmetry in \eqref{eq:symmetriesCase1}.

To show the second symmetry, let $\hat{\diagram{S}}$ be the augmented diagram of shape $\alpha$,
with an additional box $u$ in row $\sigma^{-1}_{i+1}$ and column $-1$.
The set $\diagram{B}$ is again all boxes in the basement $\sigma$ except the boxes containing $\{i,i+1\}$.
Let $\diagram{Z} = \{u\}$ in \cref{prop:latticeDiagramSymmetry}, such that the box $u$ may only contain $\{i,i+1\}$.
The non-attacking condition then forces the fillings of $\hat{\diagram{S}}$ to be of the forms in \eqref{eq:case12} ---
ignoring $u$, these fillings produce $\nonSymE_\alpha^{\sigma}$ and $\nonSymE_\alpha^{\sigma s_i}$.
\begin{equation}\label{eq:case12}
\ytableausetup{centertableaux,boxsize=1.6em}
\begin{ytableau}
\scriptstyle{\mathbf{i+1}} & \scriptstyle{\mathbf{i+1}}  &  & \cdots & \\
\none & \none[\vdots]   \\
\none & \mathbf{i}  &  &  \\
\end{ytableau}
\quad
\text{ and }
\quad
\begin{ytableau}
\mathbf{i} &\mathbf{i}  &  & \cdots & \\
\none&\none[\vdots]   \\
\none&\scriptstyle{\mathbf{i+1}}  &  &  \\
\end{ytableau}
\end{equation}
There are no extra inversions in this case, so \cref{prop:latticeDiagramSymmetry}
implies that $x_{i+1} \nonSymE_\alpha^{\sigma} +  x_i \nonSymE_\alpha^{\sigma s_i}$
is symmetric in $x_i x_{i+1}$, implying the second requirement in \eqref{eq:symmetriesCase1}.

\medskip 

\noindent
\textbf{Case $\gamma_i > \gamma_{i+1}$:} Using the exact same strategy as in previous case, we
need to show that
\begin{equation}\label{eq:symmetriesCase2}
\nonSymE_\alpha^{\sigma}(\xvec;q,t) + \nonSymE_\alpha^{\sigma s_i}(\xvec;q,t) \quad \text{ and } \quad
tx_{i+1} \cdot \nonSymE_\alpha^{\sigma}(\xvec;q,t) + x_i \cdot \nonSymE_\alpha^{\sigma s_i}(\xvec;q,t)
\end{equation}
are symmetric in $x_i x_{i+1}$.

As before, we fix the basement entries which are not in $\{i,i+1\}$ and consider fillings
of the two types in \eqref{eq:case21}.
\begin{equation}\label{eq:case21}
\ytableausetup{centertableaux,boxsize=1.6em}
\begin{ytableau}
\infty_2 & \scriptstyle{\mathbf{i+1}}  &  & \\
\none & \none[\vdots]   \\
\infty_1  & \mathbf{i}  &  & \cdots & \\
\end{ytableau}
\quad
\text{ and }
\quad
\begin{ytableau}
\infty_2 &\mathbf{i}  &  & \\
\none & \none[\vdots]   \\
\infty_1 & \scriptstyle{\mathbf{i+1}}  &  & \cdots & \\
\end{ytableau}
\end{equation}
In this case, there is no extra inversion in either of these
(we can imagine that the boxes marked $\infty_1$ and $\infty_2$ are greater than all other boxes, and $\infty_2 > \infty_1$)
so the first symmetry in \eqref{eq:symmetriesCase2} is straightforward.

Finally, as in the previous case, we add an extra box $u \in \diagram{Z}$, filled with either $i$ or $i+1$.
The fillings in \eqref{eq:case22} then give $\nonSymE_\alpha^{\sigma}(\xvec;q,t)$ and $ \nonSymE_\alpha^{\sigma s_i}(\xvec;q,t)$.
\begin{equation}\label{eq:case22}
\ytableausetup{centertableaux,boxsize=1.6em}
\begin{ytableau}
\scriptstyle{\mathbf{i+1}} & \scriptstyle{\mathbf{i+1}}  &  & \\
\none & \none[\vdots]   \\
\infty & \mathbf{i}  &  & \cdots & \\
\end{ytableau}
\quad
\text{ and }
\quad
\begin{ytableau}
\mathbf{i} & \mathbf{i}  &  & \\
\none & \none[\vdots]   \\
\infty & \scriptstyle{\mathbf{i+1}}  &  & \cdots & \\
\end{ytableau}
\end{equation}
We note that the second type has an extra inversion of type $B$, given by entries $\{i,i+1,\infty\}$
with $i$ and $\infty$ in the leftmost column. Hence,
$x_{i+1} \cdot \nonSymE_\alpha^{\sigma}(\xvec;q,t) + t^{-1}x_i \cdot \nonSymE_\alpha^{\sigma s_i}(\xvec;q,t)$
is symmetric in $x_i x_{i+1}$ which implies the last symmetry needed.

\medskip

Relation \eqref{eq:keyOperatorOnBasement} now follows from the first
by applying $\tpi_i$ on both sides of \cref{eq:atomOperatorOnBasement}
and use the fact that $\tpi_i \ttheta_i(f) = t f$ for all $f$.
\end{proof}

Repeated application of these operators gives us the following corollary:
\begin{corollary}\label{cor:twinv}
Let $\sigma$ and $\alpha$ be given and define $\twinv_\theta(\alpha, \sigma)$ and $\twinv_\pi(\alpha, \sigma)$ as
\begin{align*}
\twinv_\theta(\alpha, \sigma) &= \left| \{ (i,j) : i  < j ,\; \alpha_i \leq \alpha_j \text{ and } \sigma_i < \sigma_j \}   \right| \\
\twinv_\pi(\alpha, \sigma) &= \left| \{ (i,j) : i<j ,\; \alpha_i < \alpha_j  \text{ and } \sigma_i < \sigma_j \}   \right|
\end{align*}
Then
\begin{align}
\ttheta_\sigma \nonSymE^{\omega_0}_\alpha(\xvec;q,t) &= t^{ \twinv_\theta(\omega_0 \alpha, \sigma) }  \nonSymE^{\omega_0\sigma}_\alpha(\xvec;q,t) \\
\tpi_\sigma \nonSymE^{\id}_\alpha(\xvec;q,t)     &= t^{ \twinv_\pi(\alpha, \sigma) } \nonSymE^\sigma_\alpha(\xvec;q,t).
\end{align}
\end{corollary}
\begin{proof}
The proof is more or less immediate via induction over $\length(\sigma)$
by unraveling \eqref{eq:atomOperatorOnBasement} and \eqref{eq:keyOperatorOnBasement}.
\end{proof}

\subsection{Permuting the shape}

We now prove a more general version of an identity in \cite{HaglundNonSymmetricMacdonald2008},
where the case $\sigma = \omega_0$ is proved.

\begin{proposition}[Shape permuting operators]\label{prop:thetaShapeTrans}
If $\alpha_j < \alpha_{j+1}$, $\sigma_j = i+1$ and $\sigma_{j+1} =i$ for some $i$, $j$,
then
\begin{equation}\label{eq:macdonaldShapeTrans}
\nonSymE^\sigma_{s_j \alpha}(\xvec; q, t) = \left( \ttheta_i + \frac{1-t}{1-q^{1+\leg u}t^{\arm u}} \right) \nonSymE^\sigma_{\alpha}(\xvec; q, t),
\end{equation}
where  $u = (j+1, \alpha_{j}+1)$ in the diagram of shape $\alpha$.
\end{proposition}
\begin{proof}
The case $\sigma = \omega_0$ is our base case in an inductive argument over different basements.
It is enough to show the following equalities (for fixed $\alpha$):
\begin{enumerate}
\item Equation \eqref{eq:macdonaldShapeTrans} holds for the triple $(\sigma, i, j)$
	if and only if it holds for $(\sigma s_k, i, j)$ if $k \neq \{i-1, i, i+1\}$.

\item Equation \eqref{eq:macdonaldShapeTrans} holds for $(\sigma, i, j)$ with $(\sigma_{j-1}, \sigma_j, \sigma_{j+1}) = (i-1,i+1,i)$
if and only if it holds for $(\sigma s_{i-1} s_i, i-1, j)$.

\item Equation \eqref{eq:macdonaldShapeTrans} holds for $(\sigma, i, j)$ where $(\sigma_{j}, \sigma_{j+1}, \sigma_{j+2}) = (i+1,i,i-1)$
if and only if it holds for $(\sigma s_{i-1} s_i, i-1, j)$.
\end{enumerate}
The first equality ensures that we are free to permute the basement labels not involving $i$ and $i+1$.
The last two equalities allow us to increase (decrease) the basement labels on rows $j$, $j+1$ by one,
provided that there is also a third row where the label is decreased (increased) by two.
It is easy to see that with these operations, one can reach any configuration $(\sigma,i,j)$
satisfying the conditions in \cref{prop:thetaShapeTrans} from the base case.

\medskip

Consider the statement
\begin{equation}\label{eq:macdonaldShapeTransSimp}
\nonSymE^\sigma_{s_j \alpha}(\xvec; q, t) = \left( \ttheta_i + C_u \right) \nonSymE^\sigma_{\alpha}(\xvec; q, t),
\qquad C_u = \frac{1-t}{1-q^{1+\leg u}t^{\arm u}}.
\end{equation}
\noindent
\textbf{Case 1:} We want to show that \eqref{eq:macdonaldShapeTrans} holds
for the left configuration in \eqref{eq:casePermAbove} if and only if it holds for the right hand side.
Note that \eqref{eq:casePermAbove} only illustrate on of several possible relative positions of $k$, $k+1$ and $i$.
\begin{equation}\label{eq:casePermAbove}
\ytableausetup{centertableaux,boxsize=1.6em}
\begin{ytableau}
\mathbf{k}& \cdots \\
\none[\vdots] \\
\scriptstyle{\mathbf{i+1}} & \cdots & \\
\mathbf{i}   & \cdots &  & u & \cdots \\
\none[\vdots] \\
\scriptstyle{\mathbf{k+1}}& \cdots \\
\end{ytableau}
\qquad
\Longleftrightarrow 
\qquad
\ytableausetup{centertableaux,boxsize=1.6em}
\begin{ytableau}
\scriptstyle{\mathbf{k+1}}& \cdots \\
\none[\vdots] \\
\scriptstyle{\mathbf{i+1}} & \cdots & \\
\mathbf{i}   & \cdots &  & u & \cdots \\
\none[\vdots] \\
\mathbf{k}& \cdots \\
\end{ytableau} 
\end{equation}
We apply $\ttheta_k$ or $\tpi_k$ depending on if $k+1$ appears below or above $k$, respectively,
on both sides of \eqref{eq:macdonaldShapeTransSimp}. Both $\ttheta_k$ and $\tpi_k$
commute with $\ttheta_i$ since $|k-i| \geq 2$ and we obtain
\begin{equation}\label{eq:macdonaldShapeTransSimp2}
t^\ast \cdot \nonSymE^{\sigma s_k}_{s_j \alpha}(\xvec; q, t) = t^\ast \cdot \left( \ttheta_i + C_u \right) \nonSymE^{\sigma s_k}_{\alpha}(\xvec; q, t)
\end{equation} 
using \cref{prop:basementPermutation}. The factor $t^\ast$ depends on the relative lengths of rows with basement label $k$ and $k+1$,
but it is the same on both sides.
Since going from \eqref{eq:macdonaldShapeTransSimp}
to \eqref{eq:macdonaldShapeTransSimp2} is invertible, we have the desired equality.

\medskip 

\noindent
\textbf{Case 2:} To get from the basement $\sigma$ in the left hand side in \eqref{eq:casePermMove2} to the basement in the right hand side,
we need to perform $s_{i-1}$ followed by $s_i$ as right multiplication.
\begin{equation}\label{eq:casePermMove2}
\ytableausetup{centertableaux,boxsize=1.6em}
\begin{ytableau}
\scriptstyle{\mathbf{i-1}}& \cdots \\ 
\scriptstyle{\mathbf{i+1}} & \cdots & \\
\mathbf{i}   & \cdots &  & u & \cdots \\
\end{ytableau}
\qquad
\Longleftrightarrow 
\qquad
\ytableausetup{centertableaux,boxsize=1.6em}
\begin{ytableau}
\scriptstyle{\mathbf{i+1}}& \cdots \\ 
\mathbf{i} & \cdots & \\
\scriptstyle{\mathbf{i-1}}   & \cdots &  & u & \cdots \\
\end{ytableau}
\end{equation}
Note that $\length(\sigma) < \length(\sigma s_{i-1}) < \length(\sigma s_{i-1} s_i)$,
so to transform the basement in \eqref{eq:macdonaldShapeTransSimp} from $\sigma$ to $\sigma s_{i-1} s_i$,
we need to apply $\tpi_{i-1}$ followed by $\tpi_i$. We get
\begin{equation*}
\left(\tpi_i\tpi_{i-1}\right)\nonSymE^\sigma_{s_j \alpha}(\xvec; q, t) =
\left(\tpi_i\tpi_{i-1}\right) \left( \ttheta_i + C_u \right) \nonSymE^\sigma_{\alpha}(\xvec; q, t).
\end{equation*}
The right hand side is expanded and the square brackets has been rewritten using \cref{lem:braidRelations}:
\begin{align*}
\left(\tpi_i\tpi_{i-1}\right) \nonSymE^{\sigma}_{s_j \alpha}(\xvec; q, t) &=  \left[\ttheta_{i-1} \tpi_i\tpi_{i-1}\right] \nonSymE^\sigma_{\alpha}(\xvec; q, t)  + C_u  \left(\tpi_i\tpi_{i-1}\right) \nonSymE^{\sigma}_{\alpha}(\xvec; q, t) \\
&=  \left( \ttheta_{i-1}   + C_u \right)  \left(\tpi_i\tpi_{i-1}\right) \nonSymE^{\sigma}_{\alpha}(\xvec; q, t).
\end{align*} 
The operators $\tpi_i\tpi_{i-1}$ now act on the basement, giving a factor $t^\ast$.
Note that this factor is the same on both sides
since the comparisons performed on row lengths are the same for $\alpha$ and $s_j \alpha$.
The end result is the relation
\[
\nonSymE^{\sigma s_{i-1} s_i}_{s_j \alpha}(\xvec; q, t) =  \left[ \ttheta_{i-1}   + C_u \right]\nonSymE^{\sigma  s_{i-1} s_i}_{\alpha}(\xvec; q, t)
\]
which is what we wish to prove. Since every step has an inverse, we have the desired equivalence.

\medskip 

\noindent
\textbf{Case 3:} This case, showing the equivalence in \cref{eq:casePermMove3},
is performed in the same manner as in the previous case, now using $\ttheta_{i-1}$ followed by $\ttheta_i$ to
go from the configuration in the left hand side to the one in the right hand side.
\begin{equation}\label{eq:casePermMove3}
\ytableausetup{centertableaux,boxsize=1.6em}
\begin{ytableau}
\scriptstyle{\mathbf{i+1}} & \cdots & \\
\mathbf{i}   & \cdots &  & u & \cdots \\
\scriptstyle{\mathbf{i-1}}& \cdots \\
\end{ytableau}
\qquad
\Longleftrightarrow
\qquad
\ytableausetup{centertableaux,boxsize=1.6em}
\begin{ytableau}
\mathbf{i} & \cdots & \\
\scriptstyle{\mathbf{i-1}}   & \cdots &  & u & \cdots \\
\scriptstyle{\mathbf{i+1}}& \cdots \\
\end{ytableau} 
\end{equation}
We apply $\ttheta_i\ttheta_{i-1}$ on both sides of \eqref{eq:macdonaldShapeTransSimp} and obtain
\begin{equation*}
\left(\ttheta_i\ttheta_{i-1}\right)\nonSymE^\sigma_{s_j \alpha}(\xvec; q, t) = \left(\ttheta_i\ttheta_{i-1}\right) \left( \ttheta_i + C_u \right) \nonSymE^\sigma_{\alpha}(\xvec; q, t).
\end{equation*}
Expanding the right hand side, where the braid relation has been used in the square bracket, gives
\begin{align*}
\left(\ttheta_i\ttheta_{i-1}\right) \nonSymE^{\sigma}_{s_j \alpha}(\xvec; q, t) &=  \left[ \ttheta_{i-1} \ttheta_i\ttheta_{i-1}\right] \nonSymE^\sigma_{\alpha}(\xvec; q, t)  + C_u  \left(\ttheta_i\ttheta_{i-1}\right) \nonSymE^{\sigma}_{\alpha}(\xvec; q, t) \\
&=  \left[ \ttheta_{i-1}   + C_u \right]  \left(\ttheta_i\ttheta_{i-1}\right)  \nonSymE^{\sigma}_{\alpha}(\xvec; q, t).
\end{align*}
As in the previous case, this implies the equality.
\end{proof}

Note that for any $A$ not depending on the $x_i$, we can invert the operator $(\ttheta_i + A)$.
We have that
\[
 (\ttheta_i + A)^{-1} = \frac{(A+t-1-\ttheta_i)}{(A-1)(A-t)},
\]
which is easy to prove using $\ttheta_i^2 = (t-1)\ttheta_i + t$.  
This fact together with \cref{prop:thetaShapeTrans} implies the following proposition.

\begin{proposition}[Shape permuting operators II]\label{prop:thetaShapeTrans2}
If $\alpha_j > \alpha_{j+1}$, $\sigma_j = i+1$ and $\sigma_{j+1} =i$ for some $i$, $j$, then
\begin{equation}\label{eq:macdonaldShapeTrans2}
\nonSymE^\sigma_{s_j\alpha}(\xvec; q, t) = \frac{ (C_u + t - 1 - \ttheta_i) \nonSymE^\sigma_{\alpha}(\xvec; q, t) }{ (C_u-1)(C_u-t) }
\end{equation}
where $C_u =  \frac{1-t}{1-q^{1+\leg u}t^{1+\arm u}}$ and $u = (j, \alpha_{j+1}+1)$ in the diagram with shape $\alpha$.
\end{proposition}

These two identities tell us how $\ttheta_i$ act on non-symmetric Macdonald polynomials.
The case $\alpha_j=\alpha_{j+1}$ is discussed in the next section.
The special cases with $\sigma = \omega_0$ or $\sigma = \id$
appear in various places, \cite{mimachi1998,bakerForrester1997,Marshall1999}.

\section{Partial symmetries}

The goal of this section is to prove partial symmetries of non-symmetric Macdonald polynomials.
More specifically, if the shape $\alpha$ and basement $\sigma$ are such that the augmented diagram is of the form
\begin{equation}
\ytableausetup{centertableaux,boxsize=1.6em}
\begin{ytableau}
\none[\vdots] \\
\scriptstyle{\mathbf{i+1}} & & \cdots &   \\
\mathbf{i}                 & & \cdots &   \\
\none[\vdots]  \\
\end{ytableau}
\qquad
\text{ or }
\qquad
\begin{ytableau}
\none[\vdots] \\
\mathbf{i} & &\cdots &    \\
\scriptstyle{\mathbf{i+1}}    &    & \cdots &   \\
\none[\vdots]  \\
\end{ytableau}
\end{equation}
where two adjacent rows have equal lengths and the basement labels differ by $1$,
then the corresponding Macdonald polynomial $\nonSymE^\sigma_\alpha(\xvec;q,t)$
is symmetric in $x_i x_{i+1}$.
We prove this by showing several implications, which turns into an inductive proof.

In \cite{Marshall1999}, the polynomials $\nonSymE^{\id}_\alpha(\xvec,q,t)$ were studied.
For example, he derives the analogous formulas for the shape-permuting operators.
We need the following statement from the same section --- now translated to our notation --- which appears in \cite[Equation (3.4)]{Marshall1999}:
\begin{lemma}\label{lem:marshallAtomSymmetry}
Suppose $\alpha_i=\alpha_{i+1}$. Then $\theta_i \nonSymE^{\id}_\alpha(\xvec;q,t) = t \nonSymE^{\id}_\alpha(\xvec;q,t).$
\end{lemma}
It would be interesting go give a combinatorial
proof of this identity using \cref{prop:latticeDiagramSymmetry}.

\begin{lemma}\label{lem:symmetryImplicationsI}
Suppose $\alpha_j = \alpha_{j+1}$ and $\{\sigma_j, \sigma_{j+1} \} = \{i, i+1\}$ for some $j$, $i$.
Then the following statements are equivalent:
\begin{enumerate}
 \item $\nonSymE_\alpha^\sigma(\xvec;q,t)$ is symmetric in $x_i x_{i+1}$,
 \item $\ttheta_i \nonSymE_\alpha^\sigma(\xvec;q,t) = t \nonSymE_\alpha^\sigma(\xvec;q,t) $,
 \item $\nonSymE_\alpha^{\sigma s_i}(\xvec;q,t)$ is symmetric in $x_i x_{i+1}$,
 \item $\nonSymE_\alpha^{\sigma s_k}(\xvec;q,t)$, $k \notin \{i-1,i,i+1 \}$, is symmetric in $x_i x_{i+1}$.
\end{enumerate}
\end{lemma}
\begin{proof}
We have that $(1) \Leftrightarrow (2)$ using \cref{lem:tpisym},
$(2) \Leftrightarrow (3)$ using \cref{prop:basementPermutation} and \cref{lem:operatorSymmetry},
and finally $(1) \Leftrightarrow (4)$ by using \cref{lem:operatorSymmetry}.
\end{proof}

\begin{lemma}\label{lem:symmetryImplicationsII}
Suppose $\alpha_j = \alpha_{j+1}$ and $\{\sigma_j, \sigma_{j+1} \} = \{i, i+1\}$ for some $j$, $i\geq 2$.
Then the following statements are equivalent:
\begin{enumerate}
 \item $\nonSymE_\alpha^\sigma(\xvec;q,t)$ is symmetric in $x_i x_{i+1}$,
 \item $\nonSymE_\alpha^{\sigma s_{i-1} s_i} (\xvec;q,t)$ is symmetric in $x_{i-1} x_{i}$.
\end{enumerate}
\end{lemma}
\begin{proof}
\cref{prop:basementPermutation} implies that either
\[
\ttheta_i\ttheta_{i-1}\nonSymE_\alpha^\sigma(\xvec;q,t) = t^\ast \nonSymE_\alpha^{\sigma s_{i-1} s_i} (\xvec;q,t) \text{ or }
\tpi_i\tpi_{i-1}\nonSymE_\alpha^\sigma(\xvec;q,t) = t^\ast \nonSymE_\alpha^{\sigma s_{i-1} s_i}
\]
depending on if the basemen label $i-1$ appear earlier or later than $i$ in $\sigma$.
By linearity, it suffices to verify the stronger statement
that $\ttheta_i\ttheta_{i-1}$ and $\tpi_i\tpi_{i-1}$ maps a monomial symmetric in $x_i x_{i+1}$
to a monomial symmetric in $x_{i-1}x_i$.
This calculation is tedious, but can be verified explicitly with the definition of the Demazure--Lusztig operators.
Note that it is enough to do the computation with $\ttheta_2\ttheta_{1}$ on the monomial $x_1^a x_2^b x_3^b$,
and this can be done symbolically in a modern computer algebra system such as Mathematica.
\end{proof}

We are now ready to prove the main theorem of this section:
\begin{theorem}[Partial symmetry]\label{thm:partialSymmetry}
Suppose $\alpha_j = \alpha_{j+1}$ and that $\{\sigma_j, \sigma_{j+1} \}$ take the values $\{i, i+1\}$ for some $j$, $i$.
Then $\nonSymE_\alpha (\xvec;q,t)$ is symmetric in $x_i x_{i+1}$.
\end{theorem}
\begin{proof}

The first two items in \cref{lem:symmetryImplicationsI} together with \cref{lem:marshallAtomSymmetry}
implies that the statement is true whenever $\sigma = \id$.

We now argue in the same manner as in the proof of \cref{prop:thetaShapeTrans}:
\cref{lem:symmetryImplicationsI} together with \cref{lem:symmetryImplicationsII}
implies that we can permute the basement as long as $\sigma_j$ and $\sigma_{j+1}$ differ by one,
and still having the statement in the theorem to be true.

In other words, we can reach any basement where $\{\sigma_j, \sigma_{j+1} \}$ take the values $\{i, i+1\}$,
using the operations on the basement described in the previous two lemmas, all while preserving the symmetry property.
\end{proof}

Let $\alpha \sim \gamma$ indicate that $\alpha$ and $\gamma$ are permutations of the same partition.

\begin{corollary}\label{prop:invariantSubspace}
Fix a shape $\alpha$ and let $V$ be the subspace in $\setQ(q,t)[\xvec]$
spanned by $\{ \nonSymE^{\omega_0}_{\gamma}(\xvec; q, t) : \gamma \sim \alpha \}$.
Let $\gamma \sim \alpha$. Then
\begin{equation}
\ttheta_i\nonSymE^{\omega_0}_{\gamma}(\xvec; q, t) \in V, \; \tpi_i\nonSymE^{\omega_0}_{\gamma}(\xvec; q,t) \in V
 \text{ and } \;
 \nonSymE^\sigma_{\gamma}(\xvec; q, t) \in V,
\end{equation}
for any $i$ and $\sigma$.
\end{corollary}
\begin{proof}
It is straightforward to show that $\ttheta_i \nonSymE^{\omega_0}_{\gamma}(\xvec; q, t) \in V$:
whenever the rows with basement label $i$ and $i+1$ have different lengths the statement follows from \cref{prop:thetaShapeTrans}
or \cref{prop:thetaShapeTrans2}. In the case of equal lengths, it follows from \cref{thm:partialSymmetry}, since then
$\ttheta_i \nonSymE^{\omega_0}_{\gamma}(\xvec; q, t) = t \nonSymE^{\omega_0}_{\gamma}(\xvec; q, t)$.
This implies that the $\ttheta_i$ preserves $V$.
That $\tpi_i\nonSymE^{\omega_0}_{\gamma}(\xvec; q,t) \in V$ now follows from expressing
$\tpi_i$ in terms of $\ttheta_i$ as in \cref{eq:tpi-in-ttheta}.
Finally, the last statement is a consequence of \cref{prop:basementPermutation}
using the fact that the basement-permuting operators preserve $V$.
\end{proof}

\section{Properties of permuted basement $t$-atoms}\label{sec:tAtoms}

We define the Demazure $t$-atoms as $\atom_{\alpha}(\xvec;  t) = \nonSymE^{\id}_{\alpha}(\xvec;0,t)$
and the permuted-basement Demazure $t$-atoms as $\atom^{\sigma}_{\alpha}(\xvec;  t) = \nonSymE^{\sigma}_{\alpha}(\xvec;0,t)$.
Similarly, the $t$-key polynomials are defined as $\key_{\alpha}(\xvec;  t) = \nonSymE^{\omega_0}_{\alpha}(\xvec;0,t)$.
The $t$-atoms was previously introduced in \cite{Haglund2011463} and they have remarkable similarities with
Hall--Littlewood polynomials.

\begin{remark}
Note, the permuted-basement Demazure atoms we obtain from $\atom^{\sigma}_{\alpha}(\xvec; 0)$
\emph{do not} agree in general with the extension of Demazure atoms introduced in \cite{Haglund2012,Remmel2013}.
They impose an extra condition (called the $B$-increasing condition) on the
underlying fillings which they call \defin{permuted basement fillings (PBF)}.
One underlying reason for imposing this extra condition is to be able to do an analogue of RSK
on these fillings.
\end{remark}

\begin{lemma}
Suppose $\alpha_j < \alpha_{j+1}$.
If $\sigma_j = i+1$ and $\sigma_{j+1} =i$ for some $i$, $j$, then
\begin{equation}\label{eq:tpiOnGentAtom1}
\tpi_i \atom^\sigma_{\alpha}(\xvec; t) = \atom^\sigma_{s_i \alpha}(\xvec; t).
\end{equation}
Similarly, if $\sigma_j = i$ and $\sigma_{j+1} =i+1$ for some $i$, $j$, then
\begin{equation}\label{eq:tpiOnGentAtom2}
\tpi_i \atom^\sigma_{\alpha}(\xvec; t) = t \atom^\sigma_{s_i \alpha}(\xvec; t).
\end{equation}

\end{lemma}
\begin{proof}
To obtain \eqref{eq:tpiOnGentAtom1}, put $q=0$ in \cref{prop:thetaShapeTrans} and use the fact that $\ttheta_i + (1-t) = \tpi_i$.
The second equation is a consequence of the first as follows.
Start with $\sigma_j = i+1$ and $\sigma_{j+1} =i$  and apply $\ttheta_i$ on both sides of \eqref{eq:tpiOnGentAtom1}:
\begin{align*}
\ttheta_i \tpi_i \atom^\sigma_{\alpha}(\xvec; t) &= \ttheta_i \atom^\sigma_{s_i \alpha}(\xvec; t)
\end{align*}
The right hand side is rewritten using \eqref{eq:atomOperatorOnBasement}.
In the left hand side we use the same identity after using the fact that $\tpi_i$ and $\ttheta_i$ commutes:
\begin{align*}
\tpi_i \atom^{\sigma  s_i}_{\alpha}(\xvec; t) &= t \atom^{\sigma s_i}_{s_i \alpha}(\xvec; t).
\end{align*}
This now implies \eqref{eq:tpiOnGentAtom2}.
\end{proof}

\begin{corollary}\label{cor:tAtomKeyShapePermutingOperators}
The operators $\ttheta_i$ and $\tpi_i$ act on the shape of the $t$-atom and $t$-key in the following way:
\begin{align}
\ttheta_i \atom_{\alpha}(\xvec;  t) &= \atom_{s_i\alpha}(\xvec; t) \text{ if } \alpha_i > \alpha_{i+1}, \label{eq:atomShapePerm} \\
\tpi_i \key_{\alpha}(\xvec;  t) &= \key_{s_i\alpha}(\xvec; t) \text{ if } \alpha_i < \alpha_{i+1}.\label{eq:keyShapePerm}
\end{align}
\end{corollary}
\begin{proof}
The first statement is a direct consequence of applying $\ttheta_i$ on both sides of \eqref{eq:tpiOnGentAtom2} followed by
using $\theta_i \tpi_i = t$ and substituting $\alpha$ with $s_i \alpha$.
The second statement is \eqref{eq:tpiOnGentAtom1} with $\sigma=\omega_0$.
\end{proof}


The following two identities generalize a result which appears in \cite[Proposition 6.1]{Mason2009}.
The conclusion is that any $t$-key and $t$-atom polynomial can be obtained from a permuted-basement
$t$-atom with a \emph{partition} or \emph{reverse partition} shape, respectively.
\begin{proposition}\label{prop:tKeyAsPBF}
Let $\sigma$ be a fixed permutation, let $\lambda$ be a partition and let $\bar{\mu}$ denote the reverse of a partition $\mu$.
Then
\[
\key_{\sigma\lambda}(\xvec;t) = \atom^{\omega_0\sigma}_{\lambda}(\xvec;t)
\qquad \text{ and } \qquad
\atom_{\sigma\bar{\mu}}(\xvec;t) = \atom^\sigma_{\bar{\mu}}(\xvec;t),
\]
where $\sigma$ is the \emph{shortest permutation} taking $\lambda$ to $\sigma \lambda$,
and $\bar{\mu}$ to $\sigma \bar{\mu}$, respectively.
\end{proposition}
\begin{proof}
This follows from \cref{cor:twinv} and \cref{prop:basementPermutation} using induction over the length of $\sigma$.
We note that the identities are clearly true when $\sigma = \id$.
The first equation is now proved as follows:
We apply $\tpi_i$ on both sides, where identity \eqref{eq:tpiOnGentAtom1}
is used on the left hand side and \eqref{eq:keyOperatorOnBasement} is used on the right hand side.
A similar reasoning proves the second identity.

The condition on $\sigma$ being the shortest permutation ensures that only parts
of $\alpha$ with \emph{different} lengths are interchanged.
\end{proof}

Finally, we note that \cref{cor:tAtomKeyShapePermutingOperators} with $t=0$ implies (under the same conditions as in \cref{prop:tKeyAsPBF}) that
$\pi_\sigma \xvec^\lambda = \key_{\sigma \bar{\lambda}}(\xvec)$ and $\theta_\sigma \xvec^\mu = \atom_{\sigma \mu}(\xvec)$.
This is the standard definition of key polynomials and the
Demazure atoms, see \cite{Lascoux1990Keys,ReinerShimozono1995, Mason2009}.

\medskip

We are now ready to to prove the following proposition:
\begin{proposition}
Given $\alpha$ and $\sigma$, there is a sequence $\trho_{i_1}\cdots \trho_{i_\ell}$ such that
\begin{equation}
\atom_\alpha^\sigma(\xvec;t) = \trho_{i_1}\cdots \trho_{i_\ell} \xvec^\lambda
\end{equation}
where $\lambda$ is the partition with the parts of $\alpha$ in decreasing order
and each $\trho_{i_j}$ is one of $\ttheta_i$ or $\tpi_i$.
\end{proposition}
\begin{proof}
The case $\alpha=\lambda$ and $\sigma = \id$ is clear,
\begin{equation}\label{eq:opGenTAbase}
\atom_\lambda^{\id}(\xvec;t) = \xvec^\lambda 
\end{equation}
which follows from the triangularity property.
Using \eqref{eq:atomShapePerm} repeatedly on both sides of \eqref{eq:opGenTAbase}, we have that
\begin{equation}\label{eq:opGenTAbasement}
\atom_\alpha^{\id}(\xvec;t) = \ttheta_\tau \xvec^\lambda
\end{equation}
where $\tau$ is the shortest permutation such that $\alpha = \tau \lambda$.
Now apply a sequence of $\tpi_i$ on both sides in order to transform the basement into $\sigma$ while fixing the shape,
using the second identity in \cref{prop:basementPermutation}.
However, this will in general introduce a power of $t$, corresponding to how many times
we interchange basement labels of rows where the top row is shorter then the bottom row.
Using \cref{cor:twinv}, we obtain
\begin{equation}
t^{\twinv_\pi(\alpha,\sigma)} \atom_\alpha^{\sigma}(\xvec;t) = \tpi_\sigma \ttheta_\tau \xvec^\lambda.
\end{equation}
Note now that the word $\tpi_\sigma \ttheta_\tau$ is \emph{not reduced},
meaning that we can use the non-mixed brad relations as well as the mixed braid relations in \cref{lem:braidRelations}
together with the cancellation $\tpi_i \ttheta_i = t$.

Since the left hand side is a multiple of $t^{\twinv_\pi(\alpha,\sigma)}$,
we must have at least this number of canellations in the right hand side.
On the other hand, after these cancellations we have
\begin{equation}
\atom_\alpha^{\sigma}(\xvec;t) = t^{-\twinv_\pi(\alpha,\sigma)} \tpi_\sigma \ttheta_\tau \xvec^\lambda,
\end{equation}
where the left hand side is a non-zero polynomial when $t=0$.
Therefore, the number of cancellations must be equal to $t^{\twinv_\pi(\alpha,\sigma)}$,
giving the desired form.
\end{proof}

\section{Polynomial expansions}

As before, let $\gamma \sim \mu$ indicate that the parts of $\gamma$ is a permutation of the parts of $\mu$,
where $\gamma$ and $\mu$ are compositions with the same number of parts.

\begin{theorem}\label{thm:macdonaldPinGenMacE}
The symmetric Macdonald polynomials $\macdonaldP_\lambda(\xvec;q,t)$, indexed by partitions $\lambda$,
expands in the permuted basement Macdonald polynomials as
\begin{equation}\label{eq:macPinGenAtoms}
\macdonaldP_\lambda(\xvec;q,t) = \prod_{u \in \lambda} \left(1- q^{1+\leg u}t^{\arm u}\right) \cdot
\sum_{\gamma \sim \lambda} \frac{ t^{ \twinv_{\pi}(\gamma,\sigma) } \nonSymE^{\sigma}_{\gamma}(\xvec;q,t)}{
\prod_{v \in \gamma} \left(1- q^{1+\leg v}t^{\arm v}\right)
}.
\end{equation}
\end{theorem}
\begin{proof}
In \cite[Prop. 5.3.1]{HaglundNonSymmetricMacdonald2008}, the following expansion is obtained:
\begin{equation}\label{eq:macPinAtoms}
\macdonaldP_\lambda(\xvec;q,t) = \prod_{u \in \lambda} \left(1- q^{1+\leg u}t^{\arm u}\right) \cdot
\sum_{\gamma \sim \lambda} \frac{  \nonSymE^{\id}_{\gamma}(\xvec;q,t)}{
\prod_{v \in \gamma} \left(1- q^{1+\leg v}t^{\arm v}\right)
}.
\end{equation}
We apply $\tpi_\sigma$ on both sides. The resulting expression in the right hand side follows from \cref{cor:twinv},
and \cref{lem:operatorSymmetry} implies that $\tpi_\sigma$ acts as the identity on the symmetric polynomial in the left hand side.
\end{proof}

As a corollary, we get the following positive expansion of Hall--Littlewood polynomials in
permuted-basement $t$-atoms, by letting $q=0$ in \eqref{eq:macPinGenAtoms}.
This extends a result in \cite{Haglund2011463}.
\begin{corollary}[Hall--Littlewood in permuted-basement $t$-atoms]\label{cor:HLinGenAtoms2}
The Hall--Littlewood polynomials $\hlPolyP_\lambda(\xvec;t)$ expands positively into permuted-basement $t$-atoms:
\begin{equation}\label{eq:HLinGenAtoms2}
 \hlPolyP_\lambda(\xvec;t) = \sum_{\gamma \sim \lambda} t^{ \twinv_\pi(\gamma,\sigma) } \atom^{\sigma}_{\gamma}(\xvec;t).
\end{equation}
\end{corollary}

\medskip

Recall the classical expansion of Schur polynomials in terms of Hall--Littlewood polynomials,
\[
\schurS_\lambda(\xvec) = \sum_{\mu \vdash |\lambda|} \kfK_{\lambda\mu}(t) \hlPolyP_\mu(\xvec;t)
\]
where $\kfK_{\lambda\mu}(t)$ are the Kostka--Foulkes polynomials.
These are known to be polynomials with non-negative integer coefficients and have
a combinatorial interpretation, see \cite{LascouxSchutzenberger78}.
\cref{cor:HLinGenAtoms2} implies the following positive expansion:
\begin{corollary}[Schur in permuted-basement $t$-atoms]\label{cor:schurInGeneralatoms}
If $\lambda$ is a partition, then
\begin{equation}\label{eq:schurInGeneralatoms}
\schurS_\lambda(\xvec) = \sum_{\gamma \vdash |\lambda|} t^{ \twinv_\pi(\gamma,\sigma) } \kfK_{\lambda\gamma}(t) \atom^\sigma_\gamma(\xvec;t),
\end{equation}
where the sum now is taken over compositions of $|\lambda|$ and $\kfK_{\lambda\gamma}(t) = \kfK_{\lambda\mu}(t)$ if $\gamma \sim \mu$.

\end{corollary}
A combinatorial proof of this identity in the case $t=0$ appear in \cite{Mason2008} in the case $\sigma=\id$
and the case with a general $\sigma$ and $t=0$ will appear in \cite{Pun2016Thesis}.

\medskip

To give an overview over positive expansions of polynomials in other bases,
we present an overview in \cref{fig:positivity}.
The proofs of these expansions can be found in the references.
\begin{figure}[!ht]
\centering

\begin{tikzpicture}[xscale=3.6,yscale=2.3,scale=0.8, every node/.style={scale=0.8}]
\tikzset{
    vertex/.style = {
        draw,
	align=left,
        outer sep = 2pt,
        inner sep = 3pt,
	minimum height = 1cm
    },
    pluses/.style={
	dashed, decoration={markings,
	mark=between positions 1.5pt and 1 step 6pt with {
	\draw[-] (0,1.5pt) -- (0,-1.5pt);
       }
    },
    postaction=decorate,
  },
  posExp/.style = {thick,->,black},
}
\node[vertex] (quasiSchur)	at ( 2.5, 1) {Quasisymmetric Schur $\mathcal{S}_\alpha(\xvec)$};
\node[vertex] (atom)		at ( 1, 0) {Demazure atom $\atom_\alpha(\xvec)$};
\node[vertex] (key)		at ( 1, 2) {Key $\key_\alpha(\xvec)$};
\node[vertex] (generalAtom)	at ( 1, 1) {General atom $\atom^\sigma_\alpha(\xvec)$};
\node[vertex] (generalTAtom)	at ( 4, 1) {General $t$-atom $\atom^\sigma_\alpha(\xvec;t)$};
\node[vertex] (elem)		at ( 2.5, 4) {Elementary $e_\lambda(\xvec)$};
\node[vertex] (schurS)		at ( 2.5, 3) {Schur $\schurS_\lambda(\xvec)$};
\node[vertex] (schubert)		at ( 1, 3) {Schubert, $\mathfrak{S}_w(\xvec)$};
\node[vertex] (hallLittlewoodP)	at ( 4, 2) {Hall--Littlewood, $\hlPolyP_\lambda(\xvec;t)$};
\node[vertex] (gesselFundamental)	at ( 2.5, 0) {Gessel fundamental, $\mathrm{L}_D(\xvec)$};
\draw[posExp]  (schubert) to (key);
\draw[posExp]  (schurS) to (key);
\draw[posExp,dashed]  (key) to (generalAtom);
\draw[posExp]  (quasiSchur) to (atom);
\draw[posExp]  (schurS) to (hallLittlewoodP);
\draw[posExp,dashed]  (generalAtom) to (atom);
\draw[posExp]  (hallLittlewoodP) to (generalTAtom);
\draw[posExp]  (schurS) to (quasiSchur);
\draw[posExp]  (elem) to (schurS);
\draw[posExp]  (quasiSchur) to (gesselFundamental);
\draw[posExp,bend left,looseness=2.2]  (key) to (atom);
\end{tikzpicture}

\caption{This graph shows various families of polynomials. The arrows indicate ``expands positively in'' which means that the coefficients are
polynomials with non-negative coefficients. The proofs of the dashed edges are to appear in \cite{Pun2016Thesis}.
}\label{fig:positivity}
\end{figure}
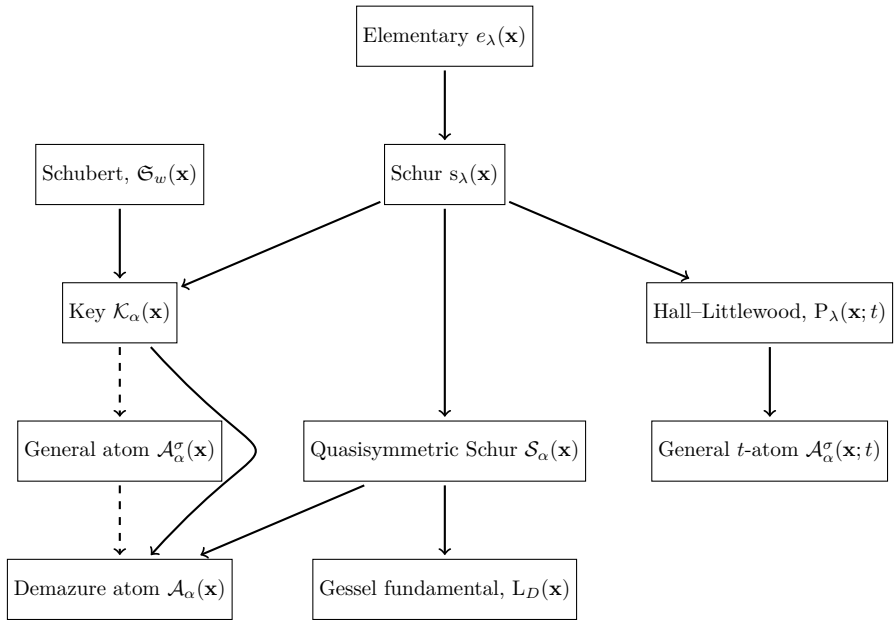

\subsection*{Acknowledgement}

The author would like to thank Jim Haglund for suggesting the topic and for providing insightful comments.
Many thanks goes to Anna Pun for interesting discussions regarding permuted-basement atoms.
This work has been funded by the \emph{Knut and Alice Wallenberg Foundation} (2013.03.07).

\section*{Appendix: Examples of permuted-basement Macdonald polynomials}

Here are some explicitly computed non-symmetric Macdonald polynomials with a permuted basement.

\begin{align}
\nonSymE_{110}^{321}(\xvec;q,t) &= \frac{(1-t)^2 x_1 x_2}{(1-q t) \left(1-q t^2\right)}+\frac{t (1-t) x_1 x_2}{1-qt^2}+\frac{(1-t) x_1 x_3}{1-q t}+x_2 x_3 \tag{$\heartsuit$}\\
\nonSymE_{110}^{312}(\xvec;q,t) &=\frac{q t (1-t)^2 x_1 x_2}{(1-q t) \left(1-q t^2\right)}+\frac{(1-t) x_1 x_2}{1-qt^2}+\frac{q (1-t) x_2 x_3}{1-q t}+x_1 x_3 \tag{$\diamondsuit$}\\
\nonSymE_{110}^{213}(\xvec;q,t) &=\frac{q (1-t)^2 x_1 x_3}{(1-q t) \left(1-q t^2\right)}+\frac{q t (1-t) x_1 x_3}{1-qt^2}+\frac{q (1-t) x_2 x_3}{1-q t}+x_1 x_2 \tag{$\spadesuit$}\\
\nonSymE_{110}^{231}(\xvec;q,t) &= \frac{q t (1-t)^2 x_1 x_3}{(1-q t) \left(1-q t^2\right)}+\frac{(1-t) x_1 x_3}{1-qt^2}+\frac{(1-t) x_1 x_2}{1-q t}+x_2 x_3 \tag{$\heartsuit$}\\
\nonSymE_{110}^{132}(\xvec;q,t) &=\frac{q (1-t)^2 x_2 x_3}{(1-q t) \left(1-q t^2\right)}+\frac{q t (1-t) x_2 x_3}{1-qt^2}+\frac{(1-t) x_1 x_2}{1-q t}+x_1 x_3 \tag{$\diamondsuit$}\\
\nonSymE_{110}^{123}(\xvec;q,t) &=\frac{q^2 t (1-t)^2 x_2 x_3}{(1-q t) \left(1-q t^2\right)}+\frac{q (1-t) x_2 x_3}{1-qt^2}+\frac{q (1-t) x_1 x_3}{1-q t}+x_1 x_2 \tag{$\spadesuit$}
\end{align}
Note that the indicated pairs of polynomials coincide when simplified.
This is a consequence of \cref{thm:partialSymmetry}.

\begin{align*}
\nonSymE_{012}^{321}(\xvec;q,t) =&\frac{q (1-t) x_2 x_3 x_1}{1-q t^2}+x_2 x_1^2,\quad
\nonSymE_{012}^{312}(\xvec;q,t) =&&\frac{q (1-t) t x_1 x_3 x_2}{1-q t^2}+x_1 x_2^2\\
\nonSymE_{012}^{213}(\xvec;q,t) =&\frac{(1-t) x_1 x_2 x_3}{1-q t^2}+x_1 x_3^2,\quad
\nonSymE_{012}^{231}(\xvec;q,t) =&&\frac{q (1-t) t x_2 x_3 x_1}{1-q t^2}+x_3 x_1^2\\
\nonSymE_{012}^{132}(\xvec;q,t) =&\frac{(1-t) x_1 x_3 x_2}{1-q t^2}+x_3 x_2^2,\quad
\nonSymE_{012}^{123}(\xvec;q,t) =&&\frac{(1-t) t x_1 x_2 x_3}{1-q t^2}+x_2 x_3^2
\end{align*}

\bibliographystyle{amsalpha}
\bibliography{bibliography}

\end{document}